\documentclass[12pt, reqno]{amsart}
\usepackage{amsmath, amsthm, amscd, amsfonts, amssymb, graphicx, color}
\usepackage[bookmarksnumbered, plainpages]{hyperref}\usepackage{cases}\usepackage{dsfont}
\numberwithin{equation}{section}
\input{amssym.def}
\input{amssym.tex}

\begin{document}

\title[Jensen functional inequality in groups]
{Ulam  stability of Jensen functional inequality in groups}

\author[G. Lu]{Gang Lu$^*$}
\address{Gang Lu \newline \indent Division of Foundational Teaching, Guangzhou College of Technology and Business, Guangzhou 510850, P.R. China}
\email{lvgang1234@163.com}\vskip 2mm

\author[L.Qu]{Lulu Qu}
\address{Lulu Qu\newline \indent Zhuhai city Polytecnic, Zhuhai, 519090,P.R. China}
\email{monaqu@126.com}\vskip 2mm
\author[Y. Jin]{Yuanfeng Jin$^*$}
\address{Yuanfeng Jin \newline \indent Department of Mathematics,  Yanbian University, Yanji 133001, P.R. China}
\email{yfkim@ybu.edu.cn}\vskip 2mm

\author[C. Park]{Choonkil Park}
\address{Choonkil Park\newline \indent  Research Institute for Natural Sciences,
Hanyang University, Seoul 04763,   Korea}
\email{baak@hanyang.ac.kr}\vskip 2mm

\begin{abstract}
In this paper we establish the Ulam stability of Jensen functional inequality in some classes of groups.
\end{abstract}


\keywords{ Jensen functional inequality; Ulam  stability; group. \\ $^*$Corresponding authors: lvgang1234@163.com (G. Lu); yfkim@ybu.edu.cn (Y. Jin).}

\theoremstyle{definition}
  \newtheorem{df}{Definition}[section]
    \newtheorem{rk}[df]{Remark}
\theoremstyle{plain}
  \newtheorem{lemma}[df]{Lemma}
  \newtheorem{theorem}[df]{Theorem}
  \newtheorem{corollary}[df]{Corollary}
    \newtheorem{proposition}[df]{Proposition}
\newtheorem{example}[df]{Example}
\setcounter{section}{0}

\maketitle

\baselineskip=14pt

\numberwithin{equation}{section}

\vskip .2in

\section{Introduction and preliminaries}

Ulam \cite{Ul}  raised a stability problem for a homomorphism in metric groups.
In 1941,
 Hyers \cite{Hy}   proved that  if  $f$ is a mapping  from a normed vector space into a Banach space and satisfies $\|f(x+y)-f(x)-f(y)\|\leq \varepsilon$, then  there exists an additive function $A$ such that $\|f(x)-A(x)\|\leq \varepsilon$.
Hyers' theorem  was generalized by Aoki \cite{A} for additive
mappings and by Rassias \cite{R2} for linear mappings by
considering an unbounded Cauchy difference.  A
generalization of the  Rassias theorem was obtained by G\u
avruta \cite{Ga} by replacing the unbounded Cauchy difference by a
general control function in the spirit of Rassias' approach.
The stability problems for several functional equations or
inequalities have been extensively investigated by a number of
authors and there are many interesting results concerning this
problem (see \cite{as, CR,  CPS, CSY, Ch, c, jpr, ljr, LP, LP1, Park1, ppp, PCH, PLZ, wa}).

In this paper we investigate the stability of the Jensen functional inequality
\begin{eqnarray*}
\|f(xy)+f(xy^{-1})-2f(x)\|\leq \|\rho(f(xy)-f(x)-f(y))\|
\end{eqnarray*}
in some classes of noncommutative groups. The Jensen functional  equation
was studied in  \cite{AJN, CENS, N}.

\section{Main results}

Suppose that $X$ is an arbitrary group and $Y$ is an arbitrary real Banach space. In this section, assume that   $X$ is an arbitrary multiplicative group and $e$ is the identity element of $X$.

 \begin{df}
 We say that a function $f:X\rightarrow Y$ is a {\it $\rho$-Jensen function} if the function satisfies
 \begin{eqnarray}\label{eqn21}
 \|f(xy)+f(xy^{-1})-2f(x)\|\leq \|\rho(f(xy)-f(x)-f(y))\|.
 \end{eqnarray}for all $x,y\in X$.
 \end{df}

 We denote the set of all $\rho$-Jensen functions by $J{\rho}(X;Y)$.
 Denote by $J\rho_0(X;Y)$ the subset of $J\rho(X;Y)$ consisting of functions
 $f$ such that $f(e)=0$. Obviously, $J\rho_0(X;Y)$ is a subspace of $J\rho(X;Y)$
 and $J\rho(X;Y)=J\rho_0(X;Y)\oplus Y$.

 \begin{df}
 We say that a function $f:X\rightarrow Y$ is an $(X,Y)$-quasi $\rho$-Jensen function if there is $c>0$ such that
 \begin{eqnarray}\label{eqn2.2}
 \|f(xy)+f(xy^{-1})-2f(x)\|\leq \|\rho (f(xy)-f(x)-f(y))\|+c
 \end{eqnarray} for all $x,y\in X$.
 \end{df}

  It is clear that the set of $(X,Y)$-quasi $\rho$-Jensen functions is a real linear  space. Denote it by $KJ_{\rho}(X;Y)$.
 From (\ref{eqn2.2}), we have
\begin{eqnarray}
\|f(y)+f(y^{-1})-2f(e)\|\leq |\rho|\|f(e)\|+c.
\end{eqnarray}
Therefore
\begin{eqnarray}\label{eqn2.4}
\|f(y)+f(y^{-1})\|\leq c_1,
\end{eqnarray}
where $c_1=c+(2+|\rho|)\|f(e)\|$. Now letting $x=y$ in (\ref{eqn2.2}), we get
\begin{eqnarray}
\|f(x^2)+f(e)-2f(x)\|\leq |\rho|\|f(x^2)-2f(x)\|+c.
\end{eqnarray}
Hence
\begin{eqnarray}\label{eqn2.6}
\|f(x^2)-2f(x)\|\leq c_2 ,
\end{eqnarray}\label{eqn2.7}
where $c_2=\frac{c+\|f(e)\|}{1-|\rho|}$. Again letting $y=x^2$ in (\ref{eqn2.2}), we get
\begin{eqnarray}
\|f(x^3)+f(x^{-1}) -2f(x)\|\leq |\rho|\|f(x^3)-f(x^2)-f(x)\|+c.
\end{eqnarray}
By (\ref{eqn2.4}) and (\ref{eqn2.7}), we obtain
\begin{eqnarray}
\|f(x^3)-3f(x)\|\leq c_3,
\end{eqnarray}
where $c_3=\frac{|\rho|c_2+c_1+c}{1-|\rho|}$.

  Let $c$ be as in (\ref{eqn2.2}) and define the set $C$ as follow:
  $C=\{c_n|n\in \mathds{N}\}$, where $c_1=(|\rho|+2)\|f(e)\|+c$, $c_2=\frac{c+\|f(e)\|}{1-|\rho|}$, $c_3=\frac{|\rho|c_2+c_1+c}{1-|\rho|}$
  and $c_n=\frac{|\rho|c_{n-1}+c_{n-2}+c_1+c}{1-|\rho|}$ if $n>3$.
  \begin{lemma}
Let $f\in KJ_\rho (X;Y)$ such that
\begin{eqnarray*}
\|f(xy)+f(xy^{-1})-2f(x)\|\leq \|f(xy)-f(x)-f(y)\|+c.
\end{eqnarray*}
Then for any $x\in X$ and any $m\in \mathds{N}$ with $m ge 2$ the following relation holds:
\begin{eqnarray}\label{eqn29}
\|f(x^m)-mf(x)\|\leq c_m.
\end{eqnarray}
  \end{lemma}

  \begin{proof}
  The proof is by induction on $m$. For $m=2$, the lemma is established in (\ref{eqn2.6}). Suppose that for $m$ the lemma has been already established, let us verify it for $m+1$. Letting $y=x^m$ in (\ref{eqn2.2}), we have
  \begin{eqnarray}
  \|f(x^{m+1})+f(x^{-(m-1)})-2f(x)\|\leq |\rho|\|f(x^m)-f(x^{m-1})-f(x)\|+c.
  \end{eqnarray}
 Then
  \begin{eqnarray}
  \begin{split}
  \;&\|f(x^{m+1})-(m+1)f(x)\|-\|f(x^{-(m-1)})+f(x^{m-1})\|\\
  \;& -\|f(x^{m-1})-(m-1)f(x)\|\\
  \;& \leq |\rho|\|f(x^{m+1})-(m+1)f(x)\|+|\rho|\|f(x^m)-mf(x)\|+c
  \end{split}
    \end{eqnarray}
   and hence
   \begin{eqnarray}
   \|f(x^{m+1})-(m+1)f(x)\|\leq \frac{|\rho|c_m+c_{m-1}+c_1+c}{1-|\rho|}=c_{m+1}.
   \end{eqnarray}
   Now the lemma is proved.
  \end{proof}

\begin{lemma}
Let $f\in KJ_\rho (X;Y)$. For any $m>1$, $k\in \mathds{N}$ and $x\in X$ we have
\begin{eqnarray}\label{eqn213}
\|f(x^{m^k})-m^kf(x)\|\leq c_m(1+m+m^2+\cdots +m^{k-1})
\end{eqnarray}
and
\begin{eqnarray}\label{eqn214}
\left\|\frac{1}{m^k}f(x^{m^k})-f(x)\right\|\leq c_m.
\end{eqnarray}
\end{lemma}

\begin{proof}
The proof will be based on mathematical induction for $k$. If $k=1$, then (\ref{eqn213}) follows from (\ref{eqn29}). Suppose that (\ref{eqn213}) for $k$ is true, let us verify it for $k+1$.  Replacing $x$ by $x^m$  in (\ref{eqn213}), we get
\begin{eqnarray}
\|f(x^{m^{k+1}})-m^kf(x^m)\|\leq c_m(1+m+m^2+\cdots +m^{k-1}).
\end{eqnarray}
Now, by (\ref{eqn29}) we get
\begin{eqnarray}
\|m^kf(x^m)-m^{k+1}f(x)\|\leq c_m m^k
\end{eqnarray}
and hence
\begin{eqnarray}
\|f(x^{m^{k+1}}-m^{k+1}f(x)\|\leq c_m (1+m+\cdots +m^k).
\end{eqnarray}
This  implies
\begin{eqnarray}
\left\|\frac{1}{m^{k+1}}f(x^{m^{k+1}})-f(x)\right\|\leq c_m(1+m+\cdots+m^k)\frac{1}{m^{k+1}}\leq c_m.
\end{eqnarray}
This completes the proof of the lemma.
\end{proof}

From (\ref{eqn214} ) it follows that the set
\begin{eqnarray*}
\left\{\frac{1}{m^k}f(x^{m^k})|k\in \mathds{N}\right\}
\end{eqnarray*}
is bounded.

Replacing $x$ by  $x^{m^n}$  in (\ref{eqn214}), we obtain
\begin{eqnarray*}
\left\|\frac{1}{m^k}f(x^{m^{n+k}})-f(x^{m^n})\right\|\leq c_m,
\end{eqnarray*}
\begin{eqnarray*}
\left\|\frac{1}{m^{n+k}}f(x^{m^{n+k}})-\frac{1}{m^n}f(x^{m^n})\right\|\leq \frac{c_m}{m^n}\rightarrow 0, \quad as \quad n\rightarrow \infty.
\end{eqnarray*}

Then, the sequence
\begin{eqnarray*}
\left\{ \frac{1}{m^k}f(x^{m^k})|k\in \mathds{N}\right\}
\end{eqnarray*}
is a Cauchy sequence. Since the space $Y$ is complete, the Cauchy sequence
has a limit and denoted by $\varphi_m(x)$. That is,
\begin{eqnarray*}
\varphi_m(x)=\lim_{k\rightarrow \infty}\frac{1}{m^k}f(x^{m^k}).
\end{eqnarray*}
By (\ref{eqn214}),
\begin{eqnarray*}
\|\varphi_m(x)-f(x)\|\leq c_m, \quad \forall x \in X.
\end{eqnarray*}

\begin{lemma}\label{lem25}
Let $f\in KJ_\rho(X;Y)$ such that
\begin{eqnarray*}
\|f(xy)+f(xy^{-1})-2f(x)\|\leq \|\rho(f(xy)-f(x)-f(y))\|+c
\end{eqnarray*}
for all $x,y\in X$.
Then for any $m\in \mathds{N}$, $\varphi_m\in KJ_\rho(X;Y)$.
\end{lemma}

 \begin{proof}
 \begin{eqnarray*}
 \begin{split}
 \;&\|\varphi_m(xy)+\varphi_m(xy^{-1})-2\varphi_m(x)\|\\
 \;& =\lim{k\rightarrow \infty}\frac{1}{m^k}\|f((xy)^{m^k})+f((xy^{-1})^{m^k})-2f(x^{m^k})\|\\
 \;& \leq \lim{k\rightarrow \infty}\frac{1}{m^k} \|\rho (f((xy)^{m^k})-f(x^{m^k})-f(y^{m^k}))\|+c\\
 \;& =\|\varphi_m(xy)-\varphi_m(x)-\varphi_m(y)\|+c
   \end{split}
 \end{eqnarray*}for all $x,y\in X$.
 \end{proof}

For all $x\in X$ we get the equation
\begin{eqnarray}
\varphi_m(x^{m^k})=m^k\varphi_m(x).
\end{eqnarray}
In fact,
\begin{eqnarray}
\begin{split}
\varphi_m(x^{m^k})\;&=\lim_{l\rightarrow \infty}\frac{1}{m^l}f((x^{m^k})^{m^l})\\
\;&=\lim_{l\rightarrow \infty}\frac{m^k}{m^{k+l}}f(x^{m^{k+l}})\\
\;& =m^k\lim_{p\rightarrow \infty}\frac{1}{m^p}f(x^{m^p})=m^k\varphi_m(x).
\end{split}
\end{eqnarray}

\begin{lemma}
If $f\in KJ_\rho (X;Y)$, then $\varphi_2=\varphi_m$ for any $m\geq 2$.
\end{lemma}
\begin{proof}
By the definition of $\varphi_2$ and $\varphi_m$, we can see that
$\varphi_2$ and $\varphi_m$ belong to $KJ_\rho(X;Y)$. We can define the function
\begin{eqnarray*}
g(x)=\lim_{k\rightarrow\infty}\frac{1}{m^k}\varphi_2(x^{m^k}).
\end{eqnarray*}
Then $g(x^{m^k})=m^kg(x)$ and $g(x^{2^k})=2^kg(x)$ for all $x\in X$ and $n\in \mathds{N}$. Then there are $d_1,d_2\in \mathds{R}_+$ such that
\begin{eqnarray*}
\|\varphi_2(x)-g(x)\|\leq d_1 \quad and \quad \|\varphi_m(x)-g(x)\|\leq d_2.
\end{eqnarray*}
Thus $\varphi_2=\varphi_m$.
\end{proof}

We denote by $B(X;Y)$ the space of all bounded functions on  group $X$ that take values in $Y$. By $PJ_\rho(X;Y)$ the set of  $(X;Y)$-pseudo-$\rho$-Jensen functions, i.e.,  $(X;Y)$-quasi-$\rho$-Jensen functions $f$ such that $f(x^n)=nf(x)$ for all  $x\in X$ and $n\in \mathds{N}$.

\begin{theorem}\label{thm27}
For an arbitrary group $G$ the following decomposition holds:
\begin{eqnarray}
KJ_\rho(X;Y)=PJ_\rho(X;Y)\oplus B(X;Y).
\end{eqnarray}
\end{theorem}
\begin{proof}
It is clear that $PJ_\rho(X;Y)$ and $B(X;Y)$ are subspaces
of $KJ_\rho(X;Y)$, and
$$PJ_\rho(X;Y)\bigcap B(X;Y)=\{0\}.$$
 Now we only
need to prove that $PJ_\rho(X;Y)\oplus B(X;Y)\subseteq KJ_\rho(X;Y)$. In fact,  if $f\in KJ_\rho(X;Y)$, then we can define the function $\widehat{f}:=\lim_{k\rightarrow \infty}
\frac{1}{2^k}f(x^{2^k})$, and  it is easy to show that $\widehat{f}(x)=\varphi_m(x)$.  Hence  $\|\widehat{f}(x)-f(x)\|=\|\varphi_2(x)-f(x)\|\leq c_2$. Thus, we have $\widehat{f}\in PJ_\rho(X;Y)$ and $\widehat{f}-f\in B(X;Y)$.
\end{proof}

\section{Ulam stability}

Suppose that $G$ is a group and $E$ is a  real Banach space.

\begin{df}
We say that the inequality (\ref{eqn21}) is stable for the pair $(X;Y)$ if for any $f:X\rightarrow Y$ satisfying
functional inequality
\begin{eqnarray}
\|f(xy)+f(xy^{-1})-2f(x)\|\leq \|f(xy)-f(x)-f(y)\|+c
\end{eqnarray}
for some $c>0$ there is a solution $j$ of the functional equation (\ref{eqn21}) such that $j(x)-f(x)\in B(X;Y)$.
\end{df}
Next,  the inequality (\ref{eqn21}) is stable on $X$ if and only if $PJ_\rho(X;Y)=J_{\rho_0}(X;Y)$.

\begin{theorem}
 The  inequality (\ref{eqn21}) is stable  if and only if $PJ_\rho(X;Y)=J_{\rho_0}(X;Y)$.
\end{theorem}

\begin{proof}
Suppose $PJ_\rho(X;Y)=J_{\rho_0}(X;Y)$. Let $f$ satisfy the inequality (\ref{eqn21}) for some $c$. By Lemma \ref{lem25} and Theorem \ref{thm27}, there is a function $\widehat{f}\in PJ_\rho(X;Y)$ such that $f-\widehat{f}$ is a bounded function on $X$.

Now suppose $PJ_\rho(X;Y)\neq J_{\rho_0}(x;Y)$. Then  we will  show that the inequality (\ref{eqn2.2}) is not stable.  Let $f\in PJ_\rho(X;Y)\backslash J_{\rho_0}(x;Y)$, by Theorem \ref{thm27}, there
are $f\in PJ_\rho(X;Y)$ and $h\in B(X;Y)$ such that
$f(x)=g(x)+h(x)$. It follows that for any $n,k\in \mathds{N}$, we get
\begin{eqnarray*}
\begin{split}
\;& f(x^{n^k})=g(x^{n^k})+h(x^{n^k}),\\
\;& n^k f(x)=n^k g(x)+h(x^{n^k}),\\
\;& f(x)=g(x)+\frac{1}{n^k}h(x^{n^k}).
\end{split}
\end{eqnarray*}
Thus, we can see that $f(x)=g(x)$ for all $x\in X$. It is a contradiction to
the assumption  on  $f$. So $PJ_\rho(X;Y)=J_{\rho_0}(X;Y)$.
\end{proof}

\medskip

\section*{Declarations}

\medskip

\noindent \textbf{Availablity of data and materials}\newline
\noindent Not applicable.

\medskip

\noindent \textbf{Conflict of interest}\newline
\noindent The authors declare that they have no competing interests.

\medskip

\noindent \textbf{Fundings}\newline
\noindent
This work was supported by National Natural Science Foundation of China (No.  11761074), Project of Jilin Science and Technology Development for Leading Talent of Science and Technology Innovation in Middle and Young and Team Project(No.20200301053RQ),  and the scientific research project of Guangzhou College of Technology and Business in 2020 (No. KA202032).

\medskip

\noindent \textbf{Acknowledgements}\newline
\noindent  Not applicable.

\medskip

\noindent \textbf{Authors' contributions}\newline
\noindent The authors equally conceived of the study, participated in its
design and coordination, drafted the manuscript, participated in the
sequence alignment, and read and approved the final manuscript.

\medskip

\bibliographystyle{amsplain}

\end{document}